\documentclass{ws-jktr}
\usepackage{hyperref}
\usepackage{float}
\floatstyle{plaintop}
\restylefloat{table}

\newcommand{\C}{\mathbb{C}}

\newcommand{\Z}{\mathbb{Z}}

\begin{document}

\markboth{Ji-young Ham, A. D. Mednykh, V. S. Petrov}
{Volumes of the hyperbolic twist knot cone-manifolds}

\catchline{}{}{}{}{}

\title{Trigonometric identities and volumes of the hyperbolic twist knot cone-manifolds}

\author{Ji-Young Ham}

\address{ Department of Science, Hongik University, \\
94 Wausan-ro, Mapo-gu, Seoul,
 121-791, Korea.\\
jiyoungham1@gmail.com}

\author{Alexander Mednykh\footnote{The author is partially supported by Laboratory of Quantum Topology of Chelyabinsk State University (Russian Federation government grant 14.Z50.31.0020).}}

\address{Sobolev Institute  of Mathematics,  pr. Kotyuga 4, Novosibirsk  630090,  \\
Novosibirsk State University,  Pirogova 2, Novosibirsk  630090,\\
Chelyabinsk State University, Bratyev Kashirinykh 129, Chelyabinsk 454001, Russia.\\
mednykh@math.nsc.ru}

\author{Vladimir Petrov}

\address{Microsoft Corporation, \\
 One Microsoft Way
Redmond, WA 98052-7329
USA. \\
vpetrov@microsoft.com}

\maketitle

\begin{abstract}
 We calculate the volumes of the hyperbolic twist knot cone-manifolds using the Schl\"{a}fli formula. Even though general ideas for calculating the volumes of cone-manifolds are around, since there is no concrete calculation written, we present here the concrete calculations.
We express the length of the singular locus in terms of the distance between the two axes fixed by two generators. In this way the calculation becomes easier than using the singular locus directly.
The volumes of the hyperbolic 
twist knot cone-manifolds simpler than Stevedore's knot are known. 
As an application, we give the volumes of the cyclic coverings over the hyperbolic twist knots.
\end{abstract}



\section{Introduction}

Thurston~\cite[Chapter 5]{T1} showed that a holonomy representation 
$h_{\infty}$ of the group of a hyperbolic knot $K$ in  $\text{PSL}(2,\C)$ can be deformed into
a one-parameter family $\{h_{\alpha}\}$ of representations  to give a corresponding one-parameter family $\{C_{\alpha}\}$ of singular complete hyperbolic manifolds, the hyperbolic \emph{cone-manifolds} of a knot $K$.   
Let $m$ be a meridian of $K$. Kojima~\cite{K1} showed further that $C_{\alpha}$ is totally determined by the action of 
$h_{\alpha}(m)$ which is a rotation of angle $\alpha$ around the fixed axis of 
$h_{\alpha}$. A point on $K$ of the cone-manifold  
$C_{\alpha}$ is in the core of a neighborhood isometric to a cylinder made of an angle 
$\alpha$ wedge by identifying the two boundaries. The $\alpha$ is called a 
\emph{cone-angle} along $K$. A point off $K$ has a neighborhood isometric to a neighborhood in $\mathbb{H}^3$.
We consider the complete hyperbolic structure on the knot complement as the cone-manifold structure of cone-angle zero.

As we mentioned, if we increase the cone-angle from zero and if we keep the angle small, we get a one-parameter family of hyperbolic cone-manifolds. 
 Similarly, for a link $K$ having $n$ components, we can get an $n$-parameter family of hyperbolic cone-manifolds of a link $K$.
In particular, for each two-bridge hyperbolic link, there exists an angle $\alpha_0 \in [\frac{2\pi}{3},\pi)$ for each link $K$ such that $C_{\alpha}$ is hyperbolic for $\alpha \in (0, \alpha_0)$, Euclidean for $\alpha=\alpha_0$, and spherical for $\alpha \in (\alpha_0, \pi]$ \cite{P2,HLM1,K1,PW}. 

Explicit volume formulae for hyperbolic cone-manifolds of knots and links are known only for a few cases. The volume formulae for hyperbolic cone-manifolds of the knot 
$4_1$~\cite{HLM1,K1,K2,MR1}, the knot $5_2$~\cite{M2}, the link $5_1^2$~\cite{MV1}, 
the link $6_2^2$~\cite{M1}, and the link $6_3^2$~\cite{DMM1} have been calculated. In~\cite{HLM2} a method of calculating the volumes of two-bridge knot cone-manifolds were introduced but without explicit formulae. 

The main purpose of the paper is to find explicit and efficient volume formula for hyperbolic twist knot cone-manifolds. The following theorem gives the formula for $T_m$ for even integers $m$. For odd integers $m$, we can replace $T_m$ by $T_{-m-1}$ as explained in Section~\ref{sec:twist}. So, the following theorem actually covers all possible hyperbolic twist knots. But for the volume formula, since the knot $T_{2n}$ has to be hyperbolic, we exclude the case when $n=0,\ -1$.

\begin{theorem}\label{thm:main}
Let $T_{2n}$ be a hyperbolic twist knot. Let $T_{2n}(\alpha)$, $0 \leq \alpha < \alpha_0$ be the hyperbolic cone-manifold with underlying space $S^3$ and with singular set $T_{2n}$ of cone-angle $\alpha$. Then the volume of $T_{2n}(\alpha)$ is given by the following formula

\begin{align*}
\text{\textnormal{Vol}} \left(T_{2n}(\alpha)\right) &= \int_{\alpha}^{\pi} \log \left|\frac{A+iV}{A-iV}\right| \: d\alpha,
\end{align*}

\noindent where 
for $A=\cot{\frac{\alpha}{2}}$, $V$ with $\text{\textnormal{Im}}(V) \leq 0$ is a zero of the complex distance polynomial $P_{2n}=P_{2n}(V,B)$ which is given recursively by 

\medskip
\begin{equation*}
P_{2n} = \begin{cases}
 \left(\left(4 B^4-8 B^2+4\right) V^2-4 B^4+8 B^2-2\right) P_{2(n-1)} -P_{2(n-2)}, \ 
\text{if $n>1$}, \\
 \left(\left(4 B^4-8 B^2+4\right) V^2-4 B^4+8 B^2-2\right) P_{2(n+1)}-P_{2(n+2)},  \
\text{if $n<-1$},
\end{cases}
\end{equation*}
\medskip

\noindent with initial conditions
\begin{equation*}
\begin{split}
P_{-2} (V,B) & =\left(2 B^2-2\right) V+2 B^2-1,\\
P_{0} (V,B) & = 1, \\    
P_{2} (V,B) & =\left(4 B^4-8 B^2+4\right) V^2+\left(2-2 B^2\right) V-4 B^4+6 B^2-1, \\
\end{split}
\end{equation*}
\noindent where $B=\cos{\frac{\alpha}{2}}$.
\end{theorem}


\section{Twist knots} \label{sec:twist}

\begin{figure} 
\begin{center}
\resizebox{4cm}{!}{\includegraphics[angle=90]{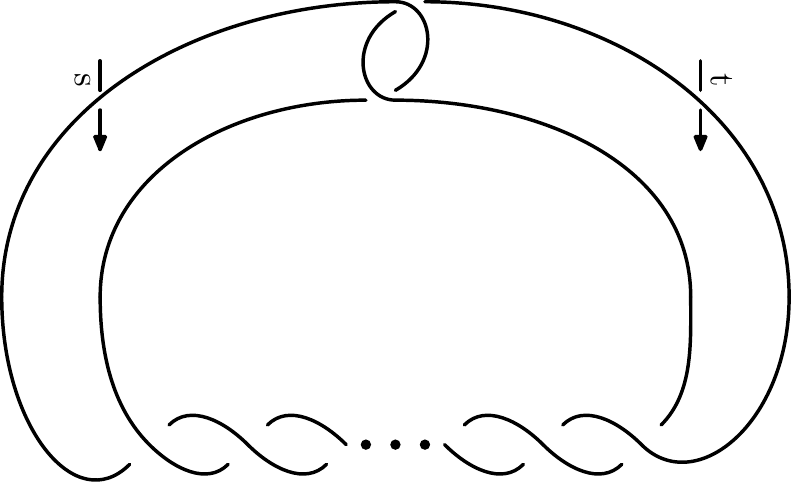}}
reflection
\reflectbox{\resizebox{4cm}{!}{\includegraphics[angle=90]{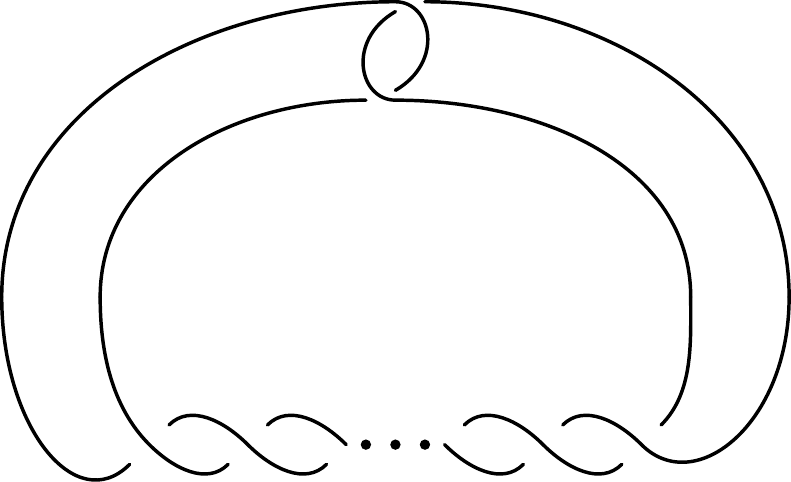}}}
\caption{A twist knot (left) and its mirror image (right)} \label{fig:T2n}
\end{center}
\end{figure}

\begin{figure}
\begin{center}
\includegraphics[angle=90]{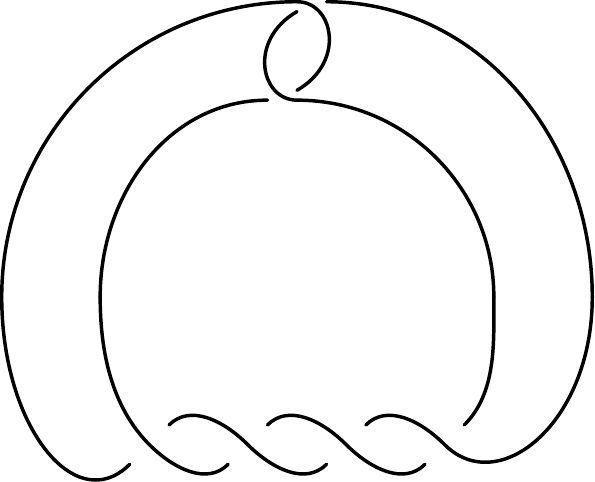}
\caption{The knot $6_1$}\label{fig:knot}
\end{center}
\end{figure}

A knot $K$ is a twist knot if $K$ has a regular two-dimensional projection of the form in Figure~\ref{fig:T2n}. For example, Figure~\ref{fig:knot} is knot $6_1$.
$K$ has 2 right-handed horizontal crossings and $m$ right-handed vertical crossings. We will denote it by $T_m$. Note that $T_m$ and its mirror image have the same fundamental group and hence have the same fundamental domain up to isometry in $\mathbb{H}^3$. It follows that $T_m(\alpha)$ and its mirror image have the same fundamental set up to isometry in $\mathbb{H}^3$ and have the same volume. So, we will make no distinction between $T_m$ and its mirror image because we are calculating volumes. Since the mirror image of $T_m$ is equivalent to $T_{-m-1}$, when $m$ is odd we will think $T_{-m-1}$ as $T_m$. Hence a twist knot can be represented by $T_{2n}$ for some integer $n$.

Let us denote by $X_{m}$ the exterior of $T_{m}$ in $S^3$. In~\cite[Proposition 1]{R1}, the fundamental group of two-bridge knots is presented. We will use the fundamental group of $X_{2n}$ in~\cite{HS}. In~\cite{HS}, the fundamental group of $X_{2n}$ is calculated with 2 left-handed horizontal crossings as positive crossings instead of two right-handed horizontal crossings. The following proposition is tailored to our purpose.

\begin{proposition}\label{thm:fundamentalGroup}
$$\pi_1(X_{2n})=\left\langle s,t \ |\ swt^{-1}w^{-1}=1\right\rangle,$$
where $w=(ts^{-1}t^{-1}s)^n$.
\end{proposition}

We remark here that $s$ of Proposition~\ref{thm:fundamentalGroup} is the meridian which winds around the bottom arc of the twist knot in Figure~\ref{fig:T2n} and $t$ is the one that does the top arc as in Figure~\ref{fig:T2n}.

\section{The complex distance polynomial and A-polynomial} \label{sec:poly}
Let $R=\text{Hom}(\pi_1 (X_{2n}), \text{SL}(2, \C))$. 
Given  a set of generators, $s,t$, of the fundamental group for 
$\pi_1 (X_{2n})$, we define
 a set $R\left(\pi_1 (X_{2n})\right) \subset \text{SL}(2, \C)^2 \subset \C^{8}$ to be the set of
 all points $(\eta(s),\eta(t))$, where $\eta$ is a
 representaion of $\pi_1 (X_{2n})$ into $\text{SL}(2, \C)$. Since the defining relation of 
 $\pi_1 (X_{2n})$ gives the defining equation of $R\left(\pi_1 (X_{2n})\right)$~\cite{R3}, $R\left(\pi_1 (X_{2n})\right)$ is an affine algebraic set in $\C^{8}$. 
$R\left(\pi_1 (X_{2n})\right)$ is well-defined up to isomorphisms which arise from changing the set of generators. We say elements in $R$ which differ by conjugations in $\text{SL}(2, \C)$ are \emph{equivalent}. 

 We use two coordinates to give the structure of the affine algebraic set to $R\left(\pi_1 (X_{2n})\right)$. Equivalently, for some $O \in \text{SL}(2, \C)$, we consider both $\eta$ and $\eta^{\prime}=O^{-1} \eta O$:

For the complex distance polynomial, we use for the coordinates

$$\eta(s)=\left[\begin{array}{cc}
  ({M+1/M})/2    &  e^{\frac{\rho}{2}}   ({M-1/M})/2     \\
       e^{-\frac{\rho}{2}} ({M-1/M})/2    &  ({M+1/M})/2
                     \end{array} \right],$$

\medskip

$$\eta(t)=\left[\begin{array}{cc}
         ({M+1/M})/2 &  e^{-\frac{\rho}{2}} ({M-1/M})/2       \\
          e^{\frac{\rho}{2}} ({M-1/M})/2    &   ({M+1/M})/2
                 \end{array}  \right],$$

and for the A-polynomial, 
\begin{center}
$$\begin{array}{ccccc}
\eta^{\prime}(s)=\left[\begin{array}{cc}
                       M &       1 \\
                        0      & M^{-1}  
                     \end{array} \right]                          
\text{,} \ \ \
\eta^{\prime}(t)=\left[\begin{array}{cc}
                   M &  0      \\
                   t      & M^{-1} 
                 \end{array}  \right].
\end{array}$$
\end{center}


\subsection{The complex distance polynomial}
Since we are interested in the excellent component (the geometric component) of $R\left(\pi_1(X_{2n})\right)$, 
in this subsection we set 
$M=e^{\frac{i \alpha}{2}}$.
Given the fundamental group of a twist knot
$$\pi_1(X_{2n})=\left \langle s,t \ |\  swt^{-1}w^{-1}=1 \right \rangle,$$
where $w=(ts^{-1}t^{-1}s)^n$, let $S=\eta(s),  T=\eta(t)$ and $W=\eta(w)$. Then the trace of $S$ and the trace of $T$ are both 
$2 \cos \frac{\alpha}{2}$. 
Let $\rho$ be the complex distance between the axes of $S$ and $T$ in the hyperbolic space $\mathbb{H}^3.$  See (\cite{F}, p. 68) for a formal definition of the complex distance (width) between oriented lines in $\mathbb{H}^3.$ The detailed formulae for calculation of $V=\cosh(\rho)$ can be found in the proof of Theorem 4.3.

\begin{lemma}\label{lem:swc}
For $c \in \text{\textnormal{SL}}(2, \C)$ which satisfies $cS=T^{-1}c$ and $c^2=-I$,
$$SWT^{-1}W^{-1}=-(SWc)^2 $$
\end{lemma}
\begin{proof}
 \begin{equation*}
\begin{split}
 (SWc)^2 & =SWcSWc=SWT^{-1}c(TS^{-1}T^{-1}S)^nc \\
              & =SWT^{-1}(S^{-1}TST^{-1})^nc^2=-SWT^{-1}W^{-1}.
 \end{split}
\end{equation*}   
\end{proof}

From the structure of the algebraic set of $R\left(\pi_1(X_{2n})\right)$ with coordinates $\eta(s)$ and $\eta(t)$ we have the defining equation of $R\left(\pi_1(X_{2n})\right)$. By plugging in $e^{\frac{i \alpha}{2}}$ into $M$ of that equation and changing the variables to $B=\cos  \frac{\alpha}{2}$ and $V=\cosh \rho$, we have the following theorem.

\begin{theorem} \label{thm:cpolynomial}
For $B=\cos  \frac{\alpha}{2}$, $V=\cosh \rho$ is a root of the following complex distance polynomial $P_{2n}=P_{2n}(V,B)$ which is given recursively by 

\medskip
\begin{equation*}
P_{2n} = \begin{cases}
 \left(\left(4 B^4-8 B^2+4\right) V^2-4 B^4+8 B^2-2\right) P_{2(n-1)} -P_{2(n-2)} \ 
\text{if $n>1$} \\
 \left(\left(4 B^4-8 B^2+4\right) V^2-4 B^4+8 B^2-2\right) P_{2(n+1)}-P_{2(n+2)}  \
\text{if $n<-1$}
\end{cases}
\end{equation*}
\medskip

\noindent with initial conditions
\begin{equation*}
\begin{split}
P_{-2} (V,B) & =\left(2 B^2-2\right) V+2 B^2-1,\\
P_{0} (V,B) & = 1, \\
P_{2} (V,B) & =\left(4 B^4-8 B^2+4\right) V^2+\left(2-2 B^2\right) V-4 B^4+6 B^2-1. \\
\end{split}
\end{equation*}
\end{theorem}

\begin{proof}
Note that $SWT^{-1}W^{-1}=I$, which gives the defining equations of 
$R\left(\pi_1(X_{2n})\right)$, is equivalent to $(SWc)^2=-I$ in $\text{SL}(2,\C)$ 
by Lemma~\ref{lem:swc} and  
$(SWc)^2=-I$ in $\text{SL}(2,\C)$ is equivalent to $\text{\textnormal{tr}}(SWc)=0$.

 We may assume
 \begin{center}
$$\begin{array}{cc}
c=\left[\begin{array}{cc}
        0 & -1      \\
        1  & 0
       \end{array}  \right],
\end{array}$$
\end{center}

 \begin{center}
$$\begin{array}{ccccc}
S=\left[\begin{array}{cc}
     \cos \frac{ \alpha}{2} & i e^{\frac{\rho}{2}} \sin \frac{ \alpha}{2}        \\
      i e^{-\frac{\rho}{2}} \sin \frac{ \alpha}{2}  &  \cos \frac{ \alpha}{2}  
                     \end{array} \right],                          
\ \ \
T=\left[\begin{array}{cc}
        \cos \frac{ \alpha}{2} & i e^{-\frac{\rho}{2}} \sin \frac{ \alpha}{2}      \\
         i e^{\frac{\rho}{2}} \sin \frac{ \alpha}{2}   &  \cos \frac{ \alpha}{2} 
                 \end{array}  \right], 
\end{array}$$
\end{center}

and let $U=TS^{-1}T^{-1}S$.

Then the equation,
\begin{align*}
\text{\textnormal{tr}}(SWc) & =\text{\textnormal{tr}}(SU^nc)=\text{\textnormal{tr}}(SU^{n-1}cU^{-1})=\text{\textnormal{tr}}(SU^{n-1}c)\text{\textnormal{tr}}(U^{-1})-\text{\textnormal{tr}}(SU^{n-1}cU) \\
            & =\text{\textnormal{tr}}(SU^{n-1}c)\text{\textnormal{tr}}(U^{-1})-\text{\textnormal{tr}}(SU^{n-2}c)=0 \ \text{if $n>1$}
\end{align*}
or
\begin{equation*}
\begin{split}
\text{\textnormal{tr}}(SWc) & =\text{\textnormal{tr}}(SU^nc)=\text{\textnormal{tr}}(SU^{n+1}cU)=\text{\textnormal{tr}}(SU^{n+1}c)\text{\textnormal{tr}}(U)-\text{\textnormal{tr}}(SU^{n+1}cU^{-1}) \\
            & =\text{\textnormal{tr}}(SU^{n+1}c)\text{\textnormal{tr}}(U)-\text{\textnormal{tr}}(SU^{n+2}c)=0 \ \text{if $n<-1$},
\end{split}
\end{equation*}
gives the complex distance polynomial, where the third equality comes from the Cayley-Hamilton theorem. By direct computations, $\text{\textnormal{tr}}(Sc)$, $\text{\textnormal{tr}}(SUc)$, and $\text{\textnormal{tr}}(SU^{-1}c)$ have $2 i \sinh \frac{\rho}{2} \sin \frac{\alpha}{2}$ as a common factor. Hence, all of $\text{\textnormal{tr}}(SWc)$'s have $2 i \sinh \frac{\rho}{2} \sin \frac{\alpha}{2}$ as a common factor. Actually, the common factor comes from the reducible representations. Just as the A-polynomials, we left the common factor out of our complex distance polynomials. We divide $\text{\textnormal{tr}}(SWc)$ by $2 i \sinh \frac{\rho}{2} \sin \frac{\alpha}{2}$ and denote $\text{\textnormal{tr}}(SWc)/(2 i \sinh \frac{\rho}{2} \sin \frac{\alpha}{2})$ by $P_{2n}$. We used \emph{Mathematica} for the calculations. 
\end{proof}

\subsection{A-polynomial}

 Let $l=w^{*}w$ and $l_*=ww^{*}$, where 
 $w^{*}$ is the word obtained by reversing $w$.
 Then
  $l$ and $l_*$ are longitudes which are null-homologus in $X_{2n}$.
 We use $l_*$ for this subsection and use both $l$ and $l_*$ in Section~\ref{sec:pytha}  to keep the original form in~\cite{M2} and to keep the familar $l_*$. One can also deal with Section~\ref{sec:pytha} with only $l_*$ or $l$.
 Define $R_U$ to be a subset of $R=\text{Hom} \left(\pi_1 (X_{2n}),\text{SL}(2,\C)\right)$ such that 
$\eta^{\prime}(l_*)$ and $\eta^{\prime}(s)$ are upper triangular. Since every representation can be conjugated into this form, any element of $R$ is equivalent to an element of $R_U$. 
By adding the equation stating that the bottom-left entry of the matrix corresponding to 
$\eta^{\prime}(l_*)$  is zero (the bottom-left entry of the matrix $\eta^{\prime}(s)$ is already zero and the equation that the bottom-left entry of the matrix corresponding to 
$\eta^{\prime}(l_*)$  is equal to zero is redundant in our setting), we have defining equations of $R_U$ and hence $R_U$ is an algebraic subset of $R$. 

Define an \emph{eigenvalue map}
\begin{center} 
 $\xi \equiv (\xi_{l_*}\times\xi_{s}):R_U\longrightarrow\C^2$
\end{center}
given by taking the top-left entries of $\eta^{\prime}(l_*)$, $L$, and of $\eta^{\prime}(s)$, $M$.
The closure of the image $\xi(C)$ of an algebraic component $C$ of $R_U$ is an algebraic subset of $\C^2$. If the closure of the image $\xi(C)$ is a curve, there is a unique defining polynomial of this curve up to constant multiples. The \emph{A-polynomial} of the knot $T_{2n}$ is defined by the product of all defining polynomials of image curves of $R_U$.  The A-polynomial of a knot can be defined up to sign~\cite{CCGLS1}. 

Practically, if we let $r=r(M,t)$ be the upper right entry of $\eta^{\prime}(sw)-\eta^{\prime}(wt)$ and $q=q(M,t)$ be the upper left entry of $\eta^{\prime}(l_*)$,
then the A-polynomial of the knot $T_{2n}$ can be obtained by taking the resultant of $M^{u_1}r$ and $M^{u_2}(q-L)$ over $t$, where the exponents $u_1$ and $u_2$ are chosen so that $M^{u_1}r$ and $M^{u_2}(q-L)$ become polynomials.

In~\cite[Theorem 1]{HS}, Hoste and Shanahan presented the A-polynomial of the twist knots.

\section{Pythagorean theorem}
\label{sec:pytha}
   Let
 $L_{\eta}=\eta(l)$ and  $L_{*\eta}=\eta(l_*)$.  If we let $l_t=l t$ and $l_s=l_* s$, then $L_T=\eta(l_t)=L_{\eta}T$, $L_S=\eta(l_s)=L_{*\eta}S$ and we have the following lemma. 
 
\begin{lemma} \label{lem:trace}
\begin{equation*}
\begin{split}
\text{\textnormal{tr}}(S^{-1}L_T) &=\text{\textnormal{tr}}(S^{-1}T) \ \text{if $n \geq 1$} \ \text{and} \\
\text{\textnormal{tr}}(T^{-1}L_S) &=\text{\textnormal{tr}}(S^{-1}T) \ \text{if $n \leq -1$}. 
\end{split}
\end{equation*}
\end{lemma} 

\begin{proof}
 Since
\begin{equation*}
\begin{split}
S^{-1}L_T &=T^{-1}S^{-1}T(ST^{-1}S^{-1}T)^{n-1} \cdot (TS^{-1})(T^{-1}STS^{-1})^{n-1}T^{-1}ST\\
           &=\left((T^{-1}STS^{-1})^{n-1}T^{-1}ST\right)^{-1}(TS^{-1})(T^{-1}STS^{-1})^{n-1}T^{-1}ST \ \text{if $n \geq 1$} \\
and\ \ \ \ \ \ &\\
T^{-1}L_S &=S^{-1}T^{-1}S(TS^{-1}T^{-1}S)^{n-1} \cdot (ST^{-1})(S^{-1}TST^{-1})^{n-1}S^{-1}TS\\
           &=\left((S^{-1}TST^{-1})^{n-1}S^{-1}TS\right)^{-1}(ST^{-1})(S^{-1}TST^{-1})^{n-1}S^{-1}TS  \ \text{if $n \leq -1$}, 
\end{split}
\end{equation*}
 we have
\begin{equation*}
\begin{split}
\text{\textnormal{tr}}(S^{-1}L_T) &=\text{\textnormal{tr}}(TS^{-1})=\text{\textnormal{tr}}(S^{-1}T) \ \text{if $n \geq 1$} \ \text{and} \\
\text{\textnormal{tr}}(T^{-1}L_S) &=\text{\textnormal{tr}}(ST^{-1})=\text{\textnormal{tr}}(TS^{-1}) =\text{\textnormal{tr}}(S^{-1}T) \ \text{if $n \leq -1$}. 
\end{split}
\end{equation*}
\end{proof}

\begin{definition} \label{def:longitude}
 The \emph{complex length} of the longitude $l$ or $l_*$ is the complex number 
 $\gamma_{\alpha}$ modulo $2 \pi \Z$ satisfying 
\begin{align*}
 \text{\textnormal{tr}}(\eta(l))=\text{\textnormal{tr}}(\eta(l_{*}))=2 \cosh \frac{\gamma_{\alpha}}{2}.
\end{align*}
 Note that 
 $l_{\alpha}=|Re(\gamma_{\alpha})|$ is the real length of the longitude of the cone-manifold $T_{2n}(\alpha)$.
\end{definition}

We prepare and prove Theorem~\ref{thm:pytha} for $T_{2n}(\alpha)$ 
with $n \geq 1$. For $n <-1$, the same Pythagrean theorem is obtained by replacing $T$ and $L_T$ with $S$ and $L_S$. We will use the oriented line matrix corresponding to a given matrix. One can refer to~\cite[Section V]{F} for oriented line matrices. Denote by $l(N)$ the line matrix corresponding to a matrix, $N$, in $\text{SL}(2,\C)$. Then $l(N)=(N-N^{-1})/\sqrt{det(N-N^{-1})}$.

By sending common fixed points of $T$ and $L_{\eta}$ to $0$ and $\infty$, we have

\begin{center}
$$\begin{array}{ccccc}
T=\left[\begin{array}{cc}
                       e^{\frac{i \alpha}{2}} & 0         \\
                        0      &  e^{-\frac{i \alpha}{2}} 
                     \end{array} \right],                          
\ \ \
L_{\eta}=\left[\begin{array}{cc}
                    e^{\frac{\gamma_{\alpha}}{2}}& 0         \\
                   0      &  e^{-\frac{\gamma_{\alpha}}{2}}
                 \end{array}  \right],
\end{array}$$
\end{center}

\begin{center}
$L_T=L_{\eta}T =\left[\begin{array}{cc}
                  e^{\frac{\gamma_{\alpha}+i \alpha}{2}}& 0                \\
                 0                & e^{-\frac{\gamma_{\alpha}+i \alpha}{2}}  
                 \end{array} \right],$ 
 \end{center}                
and the following line matrices 

\begin{align*}
l(T) &=\frac{T-T^{-1}}{2i \sinh \frac{i \alpha}{2}}  \\
        &=\frac{1}{i (e^{\frac{i \alpha}{2}}-e^{-\frac{i \alpha}{2}})}
       \left[\begin{array}{cc}
                  e^{\frac{i \alpha}{2}}-e^{-\frac{i \alpha}{2}}& 0                \\
                 0                & e^{-\frac{i \alpha}{2}}-e^{\frac{i \alpha}{2}}  
                 \end{array} \right]\\
        &=\left[\begin{array}{cc}
                  -i & 0                \\
                 0                & i
                 \end{array} \right],
 \end{align*} 

\begin{align*}
l(L_T) &=\frac{L_T-L_T^{-1}}{2i \sinh \frac{\gamma_{\alpha}+i \alpha}{2}}  \\
        &=\frac{1}{i (e^{\frac{\gamma_{\alpha}+i \alpha}{2}}-e^{-\frac{\gamma_{\alpha}+i \alpha}{2}})}
       \left[\begin{array}{cc}
                  e^{\frac{\gamma_{\alpha}+i \alpha}{2}}-e^{-\frac{\gamma_{\alpha}+i \alpha}{2}}& 0                \\
                 0  & e^{-\frac{\gamma_{\alpha}+i \alpha}{2}}-e^{\frac{\gamma_{\alpha}+i \alpha}{2}}  
                 \end{array} \right]\\
        &=\left[\begin{array}{cc}
                 -i & 0                \\
                 0                & i
                 \end{array} \right],
 \end{align*} 
which give the orientations of axes of $T$ and $L_T$.                 

\begin{figure}
\begin{center}
\includegraphics{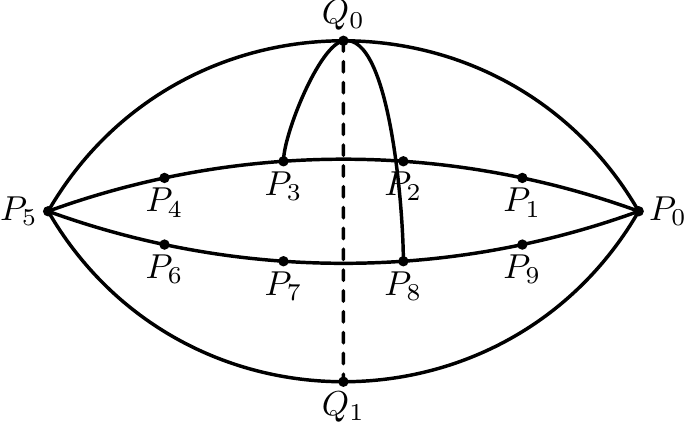}
\caption{Fundamental polyhedron for $4_1(\pi)$}\label{fig:polygon}
\end{center}
\end{figure}

Figure~\ref{fig:polygon} is the fundamental polyhedron for $T_2(\pi)$. The double branched covering space of the polyhedron along $\overline{P_3Q_0P_8}$ and $\overline{P_0Q_1P_5}$ is the lens space $L(5,3)$. The fundamental polyhedron for the hyperbolic cone-manifold of $T_2(\alpha)$ can be obtained from the fundamental polyhedron for $T_2(\pi)$ by deforming the cone-angle  continuously. Recall that a Lambert quadrangle is a quadrangle with three right angles and one acute angle, not necessarily lying on a plane. You can consult~\cite[p. 83 of Section VI]{F} or~\cite{M2} for the trigonometry of a Lambert quadrangle or a right angled hexagon. Let $m$ be the midpoint of $\overline{P_7P_8}$. Then the quadrangle $Q_0Q_1mP_8$ is a Lambert quadrangle with acute angle $\angle Q_0Q_1m=\alpha/2$, which can be considered as a right angled hexagon which is a generalized right angled triangle. The six sides are $\left(\rho,\ (\pi-\frac{\alpha}{2}) i,\ *_1,\ \frac{\pi}{2}i,\ *_2,\ \frac{\gamma_{\alpha}}{4}+ \frac{i \alpha}{2} \right)$. By applying the Law of Cosines to the hexagon, we get the formula in the following theorem geometrically and the same argument works for all twist knots. Hence, we call the following theorem Pythagorean Theorem.

Now, we are ready to prove the following theorem which gives Theorem~\ref{thm:mpytha}.
Recall that  $\gamma_{\alpha}$ modulo $2 \pi \Z$ is the \emph{complex length} of the longitude $l$ or $l_*$ of  $T_{2n}(\alpha)$.
 
\begin{theorem}(Pythagorean Theorem)\label{thm:pytha}
Let $T_{2n}(\alpha)$ be a hyperbolic cone-manifold and let $\rho$ be the complex distance between the oriented axes $S$ and $T$. 
Then we have
$$i \cosh \rho =\cot \frac{\alpha}{2} \tanh (\frac{\gamma_{\alpha}}{4}+ \frac{i \alpha}{2}).$$
\end{theorem} 

\begin{proof}
Suppose $n \geq 1$. 
\begin{align*}
\cosh \rho &=-\frac{\text{\textnormal{tr}}(l(S) l(T))}{2}\\
               &=-\frac{\text{\textnormal{tr}}(l(S) l(L_T))}{2} \\
               &=\frac{\text{\textnormal{tr}}((S-S^{-1}) (L_T-L_T^{-1}))}
               {8 \sinh \frac{i \alpha}{2} \sinh \frac{ \gamma_{\alpha}+i \alpha}{2}}\\
               &=\frac{\text{\textnormal{tr}}(SL_T-S^{-1}L_T-SL_T^{-1}+(L_TS)^{-1}))}
               {8 \sinh \frac{i \alpha}{2} \sinh \frac{ \gamma_{\alpha}+i \alpha}{2}}\\ 
               &=\frac{2 (\text{\textnormal{tr}}(SL_T)-\text{\textnormal{tr}}(S^{-1}L_T))}
               {8 \sinh \frac{i \alpha}{2} \sinh \frac{ \gamma_{\alpha}+i \alpha}{2}}\\  
               &=\frac{\text{\textnormal{tr}}(S) \text{\textnormal{tr}}(L_T) -2 \text{\textnormal{tr}}(S^{-1}L_T)}
               {4 \sinh \frac{i \alpha}{2} \sinh \frac{ \gamma_{\alpha}+i \alpha}{2}}\\ 
               &=\frac{\text{\textnormal{tr}}(S) \text{\textnormal{tr}}(L_T) -2 \text{\textnormal{tr}}(S^{-1}T)}
               {4 \sinh \frac{i \alpha}{2} \sinh \frac{ \gamma_{\alpha}+i \alpha}{2}}
               \end{align*}
where the first equality comes from ~\cite[p. 68]{F}, the sixth equality comes from the Cayley-Hamilton theorem, and the seventh equality comes from Lemma~\ref{lem:trace}.

Let $\nu = \frac{\alpha}{2}$, $\Lambda = \frac{ \gamma_{\alpha} + i \alpha}{2}$, and
$V = \cosh \rho$.
Then $\text{\textnormal{tr}}(S) = 2 \cos \nu$, $\text{\textnormal{tr}}(L_T) =  2 \cosh \Lambda$, and

\begin{align*}
\text{\textnormal{tr}}(S^{-1}T) &= tr 
    \left( \begin{pmatrix}
     \cos  \nu & -i e^{\frac{\rho}{2}} \sin \nu        \\
     - i e^{-\frac{\rho}{2}} \sin  \nu  &  \cos \nu 
                     \end{pmatrix}                          
     \begin{pmatrix}
        \cos \nu & i e^{-\frac{\rho}{2}} \sin  \nu      \\
         i e^{\frac{\rho}{2}} \sin  \nu   &  \cos  \nu
                 \end{pmatrix}  \right) \\
          &=\cos^2 \nu + e^{\rho} \sin^2 \nu + e^{- \rho} \sin^2 \nu + \cos^2 \nu\\
          &=2 ( \cos^2 \nu + V \sin^2 \nu).
\end{align*}

Hence,
\begin{equation*}
V = \frac{4 \cos \nu \cosh \Lambda -4 ( \cos^2 \nu + V \sin^2 \nu)}
               {4 i \sin \nu \sinh \Lambda}
\end{equation*}

which is equivalent to

$$V \sin \nu ( \sin \nu + i \sinh \Lambda) 
=\cos \nu (\cosh \Lambda - \cos \nu).$$

By solving for $V$, we have 

\begin{equation*}
\begin{split}
 V &=  \cot \nu \frac{\cosh \Lambda - \cos \nu}{\sin \nu + i \sinh \Lambda} \\
    &=  \cot \nu \frac{\cosh \Lambda - \cosh i \nu}
    { i \sinh \Lambda- i \sinh i \nu} \\
    &= - i \cot \nu \frac{\cosh \Lambda - \cosh i \nu}
    {\sinh \Lambda-\sinh i \nu} \\
    &= - i \cot \nu \frac{2 \sinh (\frac{\Lambda+i \nu}{2})
     \sinh (\frac{\Lambda-i \nu}{2})}
    {2 \cosh (\frac{\Lambda+i \nu}{2})
     \sinh (\frac{\Lambda-i \nu}{2})} \\
    &=- i \cot \nu \tanh \frac{\Lambda+i \nu}{2}.
 \end{split}
\end{equation*}   

By putting back $ \Lambda= \frac{ \gamma_{\alpha} + i \alpha}{2} 
= \frac{ \gamma_{\alpha}}{2} + i \nu$,
we have

$$V =- i \cot \nu \tanh ( \frac{\gamma_{\alpha}}{4} + i \nu)$$
which is equivalent to

$$i \cosh \rho =\cot \frac{\alpha}{2} \tanh (\frac{\gamma_{\alpha}}{4}+ \frac{i \alpha}{2}).$$
\end{proof}

Pythagorean theorem~\ref{thm:pytha} gives the following theorem which relates the zeros of $A_{2n}(L,M)$ and the zeros of $P_{2n}(V,B)$ for $M=e^{ \frac{i \alpha}{2}}$, 
$B=\cos \frac{\alpha}{2}$ and $A=\cot \frac{\alpha}{2}$.

\begin{theorem}\label{thm:mpytha}
Let $A=\cot \frac{\alpha}{2}$ and $M=e^{ \frac{i \alpha}{2}}$. Then the following formulae show that there is a one to one correspondence between the zeros of 
$A_{2n}(L,M)$ and the zeros of $P_{2n}(V,B)$:

$$iV=A \frac{LM^2-1}{LM^2+1} \ and \ L=M^{-2} \frac{A+iV}{A-iV}.$$
\end{theorem}

\begin{proof}
With the same notation as in the proof of Theorem~\ref{thm:pytha},
\begin{equation*}
\begin{split}
i V &= i \cosh \rho \\
     &= \cot \nu \tanh ( \frac{\gamma_{\alpha}}{4} + i \nu) \\
     &= \cot \nu \frac{\sinh ( \frac{\gamma_{\alpha}}{4} + i \nu)}
     {\cosh ( \frac{\gamma_{\alpha}}{4} + i \nu)} \\
     &= \cot \nu 
     \frac{e^{\frac{\gamma_{\alpha}}{4} + i \nu}-e^{-(\frac{\gamma_{\alpha}}{4} + i \nu)}}        {e^{\frac{\gamma_{\alpha}}{4} + i \nu}+ e^{-(\frac{\gamma_{\alpha}}{4} + i \nu)}} \\
     &= \cot \frac{\alpha}{2} 
     \frac{e^{\frac{\gamma_{\alpha}}{2} + i \alpha}-1}{e^{\frac{\gamma_{\alpha}}{2} + i \alpha}+ 1} \\
     &=A \frac{LM^2 -1}{LM^2 + 1}.
\end{split}
\end{equation*}
If we solve the above equation, $$iV=A \frac{LM^2-1}{LM^2+1},$$ for $L$, we have  
$$L=M^{-2} \frac{A+iV}{A-iV}.$$ 
\end{proof}


\section{Proof of Theorem~\ref{thm:main}} \label{sec:proof}

 We mention here that the proof can be done without referring to A-polynomial. We identified $L$ with a root of $A_{2n}(M,L)$ because A-polynomial is rather well-known.

\medskip

For $n \geq 1$ and $M=e^{i \frac{\alpha}{2}}$ ($B=\cos  \frac{\alpha}{2}$), $A_{2n}(M,L)$ and $P_{2n}(V,A)$ have $2n$ component zeros, and for  $n < -1$, $-(2n+1)$ component zeros. For each $n$, there exists an angle $\alpha_0 \in [\frac{2\pi}{3},\pi)$ such that $T_{2n}(\alpha)$ is hyperbolic for $\alpha \in (0, \alpha_0)$, Euclidean for $\alpha=\alpha_0$, and spherical for $\alpha \in (\alpha_0, \pi]$ \cite{P2,HLM1,K1,PW}. 
From the following Equality~\ref{equ:absL}, when $|L|=1$, which is equivalent to $\alpha=\alpha_0$, $\text{\textnormal{Im}}(V)=0$. Hence, when $\alpha$ increases from $0$ to $\alpha_0$, two complex numbers $V$ and $\overline{V}$ approach to a same real number. In other words, $P_{2n}(V,\cos  \frac{\alpha_0}{2})$ has a muliple root and hence $A_{2n}(L, e^{ \frac{i \alpha_0}{2}})$ has a multiple root by Theorem~\ref{thm:mpytha}.
Denote by $D(T_{2n}(\alpha))$ be the greatest common factor of the discriminant of $A_{2n}(L, e^{ \frac{i \alpha}{2}})$ over $L$  and the discriminant of
$P_{2n}(V,\cos  \frac{\alpha}{2})$ over $V$. Then $\alpha_0$ will be one of the zeros of $D(T_{2n}(\alpha))$.

From Theorem~\ref{thm:mpytha}, we have following equality,
\begin{equation}\label{equ:absL}
\begin{split}
|L|^2 &= \left|\frac{A+iV}{A-iV}\right|^2 = \frac{|A|^2+|V|^2-2A\,{\rm Im}(V)}{|A|^2+|V|^2+2A\,{\rm Im}(V)}.
\end{split}
\end{equation}

For the volume, we can either choose $|L|\geq 1$ or $|L| \leq 1$. We choose $L$ with $|L|\geq1$ and hence we have $\text{\textnormal{Im}}(V) \leq 0$ by Equality~\ref{equ:absL}. 
 Using the Schl\"{a}fli formula, we calculate the volume of 
$\eta(T_{2n})=\eta(T_{2n}(0))$ for each component with $|L|\geq1$ and having one of the zeros $\alpha_0$ of $D(T_{2n}(\alpha))$ with $\alpha_0 \in [\frac{2\pi}{3},\pi)$ on it. The component which gives the maximal volume is the excellent component~\cite{D1,FK1}. On the geometric component we have
 the volume of a hyperbolic cone-manifold 
$T_{2n}(\alpha)$ for $0 \leq \alpha < \alpha_0$:
\begin{align*}
\text{\textrm{Vol}}(T_{2n}(\alpha)) &=-\int_{\alpha_0}^{\alpha} \frac{l_{\alpha}}{2} \: d\alpha \\
                        &=-\int_{\alpha_0}^{\alpha} \log|L| \: d\alpha\\
                         &=-\int_{\pi}^{\alpha} \log|L| \: d\alpha\\
                         &=\int^{\pi}_{\alpha} \log|L| \: d\alpha\\
                         &=\int^{\pi}_{\alpha}  \log \left|\frac{A+iV}{A-iV} \right|\: d\alpha,                       
\end{align*}
where the first equality comes from the Schl\"{a}fli formula for cone-manifolds (Theorem 3.20 of~\cite{CHK}), the second equality comes from the fact that $l_{\alpha}=|Re(\gamma_{\alpha})|$ is the real length of the longitude of 
$T_{2n}(\alpha)$, the third equality comes from the fact that $\log|L|=0$  for $\alpha_0 < \alpha \leq \pi$ by Equality~\ref{equ:absL} since all the characters are real (the proof of Proposition 6.4 of~\cite{PW}) for $\alpha_0 < \alpha \leq \pi$, and 
$\alpha_0 \in [\frac{2 \pi}{3},\pi)$ is a zero of the discriminant $D(T_{2n}(\alpha))$.

We note that the fundamental set of the two-bridge link orbifolds are constructed in~\cite{MR2}. We also note that the explicit formulae for the Chern-Simons invariant of the twist knot orbifolds are presented in~\cite{HKL} and the A-polynomials of twist knots are obtained from the complex distance polynomials in~\cite{HKL}. 

\section{Volumes of the hyperbolic twist knot cone-manifolds and of its cyclic coverings}
Table~\ref{tab1} gives the approximate volume of 
$\eta(T_{2n})$ for each n between $-9$ and $9$ except the unknot and the torus knot and for each component with $\text{\textnormal{Im}}(V) \leq 0$ and having one of the zeros of 
$D(T_{2n}(\alpha))$ with $\alpha_0 \in [\frac{2\pi}{3},\pi)$ on it. We used Simpson's rule for the approximation with $10^4$ intervals from $0$ to $\alpha_0$. In that way our approximate volume on the geometric component is the same as that of SnapPea up to four decimal points. The geometric volume is written one more time on the rightmost column.

Table~\ref{table2-1} (resp. Table~\ref{table2-2}) gives the approximate volume of the hyperbolic twist knot cone-manifold, $V \left(T_{2n} (\frac{2 \pi}{k})\right)$ for $n$ between $1$ and $9$ (resp. for $n$ between $-9$ and $-2$) and for $k$ between $3$ and $10$, and of its cyclic covering, $V \left(M_k (T_{2n})\right)$. We again used Simpson's rule for the approximation with $10^4$ intervals from $\frac{2 \pi}{k}$ to $\alpha_0$. We used \emph{Mathematica} for the calculations.

\begin{table}
\tbl{ The volume of 
$\eta(T_{2n})$ for each component with $\text{\textnormal{Im}}(V) \leq 0$ and having one of the zeros of $D(T_{2n}(\alpha))$ with $\alpha_0 \in [\frac{2\pi}{3},\pi)$ on it. $N$ refers to the number of zeros of $D(T_{2n}(\alpha))$ in $[\frac{2\pi}{3},\pi)$ and $Z$ refers to the zeros of $D(T_{2n}(\alpha))$ in $[\frac{2\pi}{3},\pi)$. The volume of $T_{2n}$, $V\left(T_{2n}\right)$, is written one more time on the rigthtmost column.}
{\begin{tabular}{cc} 
\begin{tabular}{|c|c|c|c|c|}
\hline
2n & N & Z & $V\left(\eta(T_{2n})\right)$ & $V(T_{2n})$ \\
\hline
 2 & 1 & 2.0944 & 2.02988 & 2.02988 \\
\hline
 4 & 1 & 2.57414 & 3.16396 & 3.16396 \\
\hline
 6 & 1 & 2.75069 & 3.4272 & 3.4272 \\
\hline
 8 & 2 & 2.84321 & 3.52619 & 3.52619 \\

  &  & 2.3287 & 3.09308 &  \\
\hline
 10 & 2 & 2.90026 & 3.57388 & 3.57388 \\

  &  & 2.48721 & 3.29551 &  \\
\hline
 12 & 3 & 2.93897 & 3.60046 & 3.60046 \\

  &  & 2.59356 & 3.40614 &  \\

  &  & 2.22905 & 3.06705 &  \\
\hline
 14 & 3 & 2.96697 & 3.61679 & 3.61679 \\

  &  & 2.67 & 3.47332 &  \\

  &  & 2.35895 & 3.22608 &  \\
\hline
 16 & 4 & 2.98817 & 3.62753 & 3.62753 \\

  &  & 2.72765 & 3.5172 &  \\

  &  & 2.45606 & 3.3286 &  \\

  &  & 2.1754 & 3.05359 &  \\
\hline
 18 & 4 & 3.00477 & 3.63497 & 3.63497 \\

  &  & 2.7727 & 3.54747 &  \\

  &  & 2.53152 & 3.39869 &  \\

  &  & 2.28381 & 3.18371 &  \\
\hline
\end{tabular}
&
\begin{tabular}{|c|c|c|c|c|}
\hline
2n & N & Z & $V\left(\eta(T_{2n})\right)$ & $V(T_{2n})$ \\
\hline
 -2 & 1 & \multicolumn{3}{|c|}{a torus knot} \\
\hline
 -4 & 1 & 2.40717 & 2.82812 & 2.82812 \\
\hline
 -6 & 1 & 2.67879 & 3.33174 & 3.33174 \\
\hline
 -8 & 2 & 2.80318 & 3.48666 & 3.48666 \\

  &  & 2.21583 & 2.92126 &  \\
\hline
 -10 & 2 & 2.87475 & 3.55382 & 3.55382 \\

  &  & 2.41665 & 3.21098 &  \\
\hline
 -12 & 3 & 2.9213 & 3.58891 & 3.58891 \\

  &  & 2.54513 & 3.35826 &  \\

  &  & 2.14593 & 2.95204 &  \\
\hline
 -14 & 3 & 2.95401 & 3.60954 & 3.60954 \\

  &  & 2.63466 & 3.44354 &  \\

  &  & 2.29908 & 3.15591 &  \\
\hline
 -16 & 4 & 2.97825 & 3.62268 & 3.62268 \\

  &  & 2.70071 & 3.49742 &  \\

  &  & 2.41076 & 3.28251 &  \\

  &  & 2.10986 & 2.96721 &  \\
\hline
 -18 & 4 & 2.99694 & 3.63157 & 3.63157 \\

  &  & 2.75147 & 3.53365 &  \\

  &  & 2.49601 & 3.36675 &  \\

  &  & 2.2329 & 3.12459 &  \\
\hline
\end{tabular}
\end{tabular}}
\label{tab1}
\end{table}
\begin{table} \small
\caption{Volume of the hyperbolic twist knot cone-manifold, $V \left(T_{2n} (\frac{2 \pi}{k})\right)$ for $n$ between $1$ and $9$ and for $k$ between $1$ and $10$, and of its cyclic covering, $V \left(M_k (T_{2n})\right)$.
\medskip}
\begin{center}
\begin{tabular}{ll}
\begin{tabular}{|c|c|c|}
\hline
 $k$ & $V \left(T_2( \frac{2 \pi }{k})\right)$ & $V \left( M_k( T_2)\right)$ \\
\hline
 3 & \multicolumn{2}{|c|}{Euclidean}\\
\hline
 4 & 0.507471 & 2.02988 \\
 5 & 0.937207 & 4.68603 \\
 6 & 1.22129 & 7.32772 \\
 7 & 1.41175 & 9.88228 \\
 8 & 1.54386 & 12.3509 \\
 9 & 1.6386 & 14.7474 \\
 10 & 1.70857 & 17.0857\\
\hline
\end{tabular}
&
\begin{tabular}{|c|c|c|}
\hline
 $k$ & $V \left(T_4(\frac{2 \pi}{k})\right)$ & $V \left(M_k( T_4)\right)$ \\
\hline
  3 & 0.654246 & 1.96274 \\
 4 & 1.64974 & 6.59895 \\
 5 & 2.1789 & 10.8945 \\
 6 & 2.47479 & 14.8488 \\
 7 & 2.65528 & 18.587 \\
 8 & 2.77325 & 22.186 \\
 9 & 2.85453 & 25.6908 \\
 10 & 2.91289 & 29.1289 \\
\hline
\end{tabular}
\end{tabular}

\bigskip

\begin{tabular}{ll}
\begin{tabular}{|c|c|c|}
\hline
 $k$ &  $V \left(T_6(\frac{2 \pi}{k})\right)$ & $V \left(M_k (T_6)\right)$ \\
\hline
 3 & 1.13433 & 3.40299 \\
 4 & 2.1114 & 8.4456 \\
 5 & 2.5728 & 12.864 \\
 6 & 2.82779 & 16.9667 \\
 7 & 2.98374 & 20.8862 \\
 8 & 3.08602 & 24.6882 \\
 9 & 3.15669 & 28.4102 \\
 10 & 3.20752 & 32.0752 \\
\hline
\end{tabular}
&
\begin{tabular}{|c|c|c|}
\hline
 $k$ &   $V \left(T_8(\frac{2 \pi}{k})\right)$ & $V \left(M_k (T_8)\right)$ \\
\hline
 3 & 1.39476 & 4.18428 \\
 4 & 2.29275 & 9.17098 \\
 5 & 2.72019 & 13.6009 \\
 6 & 2.95903 & 17.7542 \\
 7 & 3.10589 & 21.7413 \\
 8 & 3.20251 & 25.6201 \\
 9 & 3.26938 & 29.4244 \\
 10 & 3.31754 & 33.1754 \\
\hline
\end{tabular}
\end{tabular}

\bigskip

\begin{tabular}{ll}
\begin{tabular}{|c|c|c|}
\hline
 $k$ &  $V \left(T_{10} (\frac{2 \pi}{k})\right)$ & $V \left( M_k (T_{10})\right)$ \\
\hline
 3 & 1.52984 & 4.58951 \\
 4 & 2.37879 & 9.51514 \\
 5 & 2.79028 & 13.9514 \\
 6 & 3.02165 & 18.1299 \\
 7 & 3.16431 & 22.1502 \\
 8 & 3.25831 & 26.0664 \\
 9 & 3.32342 & 29.9108 \\
 10 & 3.37035 & 33.7035 \\
\hline
\end{tabular}
&
\begin{tabular}{|c|c|c|}
\hline
 $k$ &  $V \left(T_{12} (\frac{2 \pi}{k})\right)$ & $V \left(M_k (T_{12})\right)$ \\
\hline
 3 & 1.60485 & 4.81454 \\
 4 & 2.42618 & 9.70474 \\
 5 & 2.82906 & 14.1453 \\
 6 & 3.05637 & 18.3382 \\
 7 & 3.19675 & 22.3773 \\
 8 & 3.28932 & 26.3145 \\
 9 & 3.35347 & 30.1812 \\
 10 & 3.39973 & 33.9973 \\
\hline
\end{tabular}
\end{tabular}

\bigskip

\begin{tabular}{ll}
\begin{tabular}{|c|c|c|}
\hline
 $k$ &  $V \left(T_{14} (\frac{2 \pi}{k})\right)$ & $V \left(M_k (T_{14})\right)$ \\
\hline
 3 & 1.65032 & 4.95096 \\
 4 & 2.45507 & 9.82028 \\
 5 & 2.85276 & 14.2638 \\
 6 & 3.07763 & 18.4658 \\
 7 & 3.21663 & 22.5164 \\
 8 & 3.30832 & 26.4666 \\
 9 & 3.37191 & 30.3472 \\
 10 & 3.41775 & 34.1775 \\
\hline
\end{tabular}
&
\begin{tabular}{|c|c|c|}
\hline
 $k$ &  $V \left(T_{16} (\frac{2 \pi}{k})\right)$ & $V \left(M_k (T_{16})\right)$ \\
\hline
 3 & 1.67992 & 5.03976 \\
 4 & 2.47398 & 9.89592 \\
 5 & 2.86831 & 14.3415 \\
 6 & 3.09158 & 18.5495 \\
 7 & 3.22967 & 22.6077 \\
 8 & 3.3208 & 26.5664 \\
 9 & 3.38401 & 30.4561 \\
 10 & 3.42959 & 34.2959 \\
\hline
\end{tabular}
\end{tabular}

\bigskip

\begin{tabular}{|c|c|c|}
\hline
 $k$ &  $V \left(T_{18} (\frac{2 \pi}{k})\right)$ & $V \left(M_k (T_{18})\right)$ \\
\hline
 3 & 1.70026 & 5.10079 \\
 4 & 2.48704 & 9.94814 \\
 5 & 2.87906 & 14.3953 \\
 6 & 3.10124 & 18.6074 \\
 7 & 3.23869 & 22.6708 \\
 8 & 3.32948 & 26.6358 \\
 9 & 3.39265 & 30.5339 \\
 10 & 3.43779 & 34.3779 \\
\hline
\end{tabular}
\end{center}
\label{table2-1}
\end{table}
\begin{table}
\caption{Volume of the hyperbolic twist knot cone-manifold, $V \left(T_{2n} (\frac{2 \pi}{k})\right)$ for $n$ between $-9$ and $-2$ and for $k$ between $1$ and $10$, and of its cyclic covering, $V \left(M_k (T_{2n})\right)$.
\medskip}
\label{table2-2}
\begin{tabular}{ll}
\begin{tabular}{|c|c|c|}
\hline
 $k$ &  $V \left(T_{-4} (\frac{2 \pi}{k})\right)$ & $V \left(M_k (T_{-4})\right)$ \\
\hline
 3 & 0.314236 & 0.942707 \\
 4 & 1.18737 & 4.7495 \\
 5 & 1.72248 & 8.61241 \\
 6 & 2.04253 & 12.2552 \\
 7 & 2.24401 & 15.7081 \\
 8 & 2.37774 & 19.022 \\
 9 & 2.47065 & 22.2358 \\
 10 & 2.53766 & 25.3766 \\
\hline
\end{tabular}

&
\begin{tabular}{|c|c|c|}
\hline
 $k$ &  $V \left(T_{-6}(\frac{2 \pi}{k})\right)$ & $V \left(M_k (T_{-6})\right)$ \\
\hline
 3 & 0.927278 & 2.78183 \\
 4 & 1.93717 & 7.74869 \\
 5 & 2.42921 & 12.1461 \\
 6 & 2.70001 & 16.2001 \\
 7 & 2.865 & 20.055 \\
 8 & 2.97296 & 23.7837 \\
 9 & 3.04743 & 27.4269 \\
 10 & 3.10095 & 31.0095 \\
\hline
\end{tabular}
\end{tabular}

\bigskip

\begin{tabular}{ll}
\begin{tabular}{|c|c|c|}
\hline
 $k$ &  $V \left(T_{-8}(\frac{2 \pi}{k})\right)$ & $V \left(M_k (T_{-8})\right)$ \\
\hline
 3 & 1.28595 & 3.85786 \\
 4 & 2.22061 & 8.88245 \\
 5 & 2.6616 & 13.308 \\
 6 & 2.9068 & 17.4408 \\
 7 & 3.05724 & 21.4007 \\
 8 & 3.15609 & 25.2487 \\
 9 & 3.22445 & 29.0201 \\
 10 & 3.27366 & 32.7366 \\
\hline
\end{tabular}
&
\begin{tabular}{|c|c|c|}
\hline
 $k$ &  $V \left(T_{-10}(\frac{2 \pi}{k})\right)$ & $V \left(M_k (T_{-10})\right)$ \\
\hline
 3 & 1.47286 & 4.41857 \\
 4 & 2.34274 & 9.37094 \\
 5 & 2.76087 & 13.8044 \\
 6 & 2.99535 & 17.9721 \\
 7 & 3.13977 & 21.9784 \\
 8 & 3.23486 & 25.8789 \\
 9 & 3.3007 & 29.7063 \\
 10 & 3.34815 & 33.4815 \\
\hline
\end{tabular}
\end{tabular}

\bigskip

\begin{tabular}{ll}
\begin{tabular}{|c|c|c|}
\hline
 $k$ &  $V \left(T_{-12}(\frac{2 \pi}{k})\right)$ & $V \left(M_k (T_{-12})\right)$ \\
\hline
 3 & 1.57236 & 4.71709 \\
 4 & 2.40564 & 9.62256 \\
 5 & 2.81223 & 14.0612 \\
 6 & 3.0413 & 18.2478 \\
 7 & 3.18267 & 22.2787 \\
 8 & 3.27585 & 26.2068 \\
 9 & 3.34042 & 30.0638 \\
 10 & 3.38696 & 33.8696 \\
\hline
\end{tabular}

&
\begin{tabular}{|c|c|c|}
\hline
 $k$ &  $V \left(T_{-14}(\frac{2 \pi}{k})\right)$ & $V \left(M_k (T_{-14})\right)$ \\
\hline
 3 & 1.63018 & 4.89055 \\
 4 & 2.44226 & 9.76903 \\
 5 & 2.84224 & 14.2112 \\
 6 & 3.06819 & 18.4091 \\
 7 & 3.2078 & 22.4546 \\
 8 & 3.29988 & 26.3991 \\
 9 & 3.36371 & 30.2734 \\
 10 & 3.40974 & 34.0974 \\
\hline
\end{tabular}
\end{tabular}

\bigskip

\begin{tabular}{ll}
\begin{tabular}{|c|c|c|}
\hline
$k$ &  $V \left(T_{-16}(\frac{2 \pi}{k})\right)$ & $V \left(M_k (T_{-16})\right)$ \\
\hline
 3 & 1.66659 & 4.99977 \\
 4 & 2.46545 & 9.8618 \\
 5 & 2.86129 & 14.3065 \\
 6 & 3.08528 & 18.5117 \\
 7 & 3.22379 & 22.5665 \\
 8 & 3.31517 & 26.5213 \\
 9 & 3.37856 & 30.407 \\
 10 & 3.42425 & 34.2425 \\
\hline
\end{tabular}
&
\begin{tabular}{|c|c|c|}
\hline
$k$ &  $V \left(T_{-18}(\frac{2 \pi}{k})\right)$ & $V \left(M_k (T_{-18})\right)$ \\
\hline
 3 & 1.69098 & 5.07294 \\
 4 & 2.48107 & 9.92429 \\
 5 & 2.87415 & 14.3707 \\
 6 & 3.09683 & 18.581 \\
 7 & 3.23457 & 22.642 \\
 8 & 3.32549 & 26.6039 \\
 9 & 3.38869 & 30.4982 \\
 10 & 3.43405 & 34.3405 \\
\hline
\end{tabular}
\end{tabular}
\end{table}

\section*{Acknowledgments}
The authors would like to thank Prof. Hyuk Kim for his various helps and anonymous referees for their careful suggestions.

\end{document}